\documentclass[10pt,reqno]{amsart}
\usepackage{amsmath,amssymb,mathrsfs}
\usepackage{xcolor}
\colorlet{mdtRed}{red!50!black}
\definecolor{dblue}{rgb}{0,0,.6}
\usepackage[colorlinks]{hyperref}
\hypersetup{colorlinks,linkcolor={red},citecolor={blue},urlcolor={red}}
\usepackage[all]{xy}
\usepackage{tikz,tikz-cd,tkz-graph,enumerate}

\DeclareMathOperator{\Pic}{\textnormal{Pic}}
\DeclareMathOperator{\codim}{\textnormal{codim}}

\DeclareMathOperator{\Hom}{{\rm Hom}}

\DeclareMathOperator{\Br}{\textnormal{Br}}
\DeclareMathOperator{\Sp}{\mathrm{Sp}}
\DeclareMathOperator{\malpha}{M_{\boldsymbol{\alpha}}}
\DeclareMathOperator{\mbeta}{M_{\boldsymbol{\beta}}}
\DeclareMathOperator{\ualpha}{U_{\boldsymbol{\alpha}}}
\DeclareMathOperator{\ubeta}{U_{\boldsymbol{\beta}}}
\DeclareMathOperator{\malphasm}{M^{sm}_{\boldsymbol{\alpha}}}
\DeclareMathOperator{\mbetasm}{M^{sm}_{\boldsymbol{\beta}}}

\newcommand{\mc}[1]{\mathcal{#1}}
\newcommand{\bb}[1]{\mathbb{#1}}

\newcommand{\mr}[1]{\mathrm{#1}}

\newtheorem{theorem}{Theorem}[section] 
\newtheorem{lemma}[theorem]{Lemma} 
\newtheorem{proposition}[theorem]{Proposition}
\newtheorem{corollary}[theorem]{Corollary}

\theoremstyle{definition}
\newtheorem{definition}[theorem]{Definition}
\newtheorem{remark}[theorem]{Remark}

\numberwithin{equation}{section}

\begin{document}
	
	\baselineskip=15.5pt 
	
	\title[Brauer group of moduli of parabolic symplectic bundles]{Brauer group of moduli of parabolic symplectic bundles}
	
	\author[I. Biswas]{Indranil Biswas}
	
	\address{Department of Mathematics, Shiv Nadar University, NH91, Tehsil
		Dadri, Greater Noida, Uttar Pradesh 201314, India}
	
	\email{indranil.biswas@snu.edu.in, indranil29@gmail.com}
	
	\author[S. Chakraborty]{Sujoy Chakraborty}
	
	\address{Department of Mathematics, 
		Indian Institute of Science Education and Research Tirupati, Andhra Pradesh 517507, India.}
	\email{sujoy.cmi@gmail.com}
	
	\author[A. Dey]{Arijit Dey}

	\address{Department of Mathematics, Indian Institute of Technology Chennai, Chennai, India}
	\email{arijitdey@gmail.com}
	
	\thanks{Corresponding author: Sujoy Chakraborty}
	\thanks{\textit{Email address}: sujoy.cmi@gmail.com}
	
	\begin{abstract}
		Let $X$ be a smooth connected complex projective curve of genus $g$, with $g\,\geq\, 3$. Fix an integer $r\geq 2$, a finite
		subset $D\, \subset\, X$, and a line bundle $L$ on $X$. We compute the Brauer group of the smooth locus of the moduli space of parabolic 
		symplectic stable bundles of rank $r$ on $X$ equipped with a symplectic form taking values in $L(D)$, where $L(D)$ is given the trivial
		parabolic structure.
		
	\end{abstract}
	
	\maketitle
	
	\section{introduction}
	
	Let $Y$ be a smooth quasi-projective variety over\ $\bb{C}$. The cohomological Brauer group of $Y$ is defined to be the torsion part 
	$H^2_{\text{\'et}}(Y,\,\bb{G}_m)_{tor}$. When $Y$ is smooth, it is known that $H^2_{\text{\'et}}(Y,\,\bb{G}_m)$ is actually torsion. There is 
	an equivalent formulation of Brauer groups for smooth quasi-projective varieties as the group of Morita equivalence classes of Azumaya 
	algebras, which can also be thought of as Brauer-Severi schemes (i.e., a projective bundles) on $Y$ in the \'etale topology.
	
	Parabolic vector bundles over a smooth connected projective curve $X$ were introduced by Mehta and Seshadri \cite{MeSe80} in order to generalize 
	the Narasimhan-Seshadri theorem to the case of punctured Riemann surfaces. A parabolic vector bundle, denoted by $E_*$, is a vector bundle $E$ on $X$ 
	together with the data of a filtration on the fibers of $E$ over a fixed finite subset $D$ of $X$, and certain increasing 
	sequence of real numbers, called weights, associated to these filtrations. The filtration data also provide a partition of 
	$\text{rank}(E)$ into a set of positive integers, usually called as multiplicities, at each point of $D$. Let $G$ be a connected complex 
	reductive group. The notion of parabolic vector bundles was generalized to the context of principal $G$-bundles in \cite{BhRa89}. Here, we 
	take $G$ to be the symplectic group $\Sp(r,\bb{C})$, where $r$ is an even positive integer. A parabolic $\Sp(r,\bb{C})$--bundle can also 
	be thought of as a parabolic vector bundle of rank $r$ together with a nondegenerate alternating bilinear form taking values in a 
	parabolic line bundle (cf. \cite[Definition 2.1]{BiMaWo11}).
	
	Here the setup is as follows. Let $X$ be a smooth connected complex projective algebraic curve of genus $g$, with $g\,\geq\, 3$. Fix an even positive integer $r\,\geq\, 2$, a finite subset $D=\{p_1,p_2,\cdots,p_n\}\, \subset\, X$, and a line bundle $L$ on $X$. Fix a system of multiplicities $\boldsymbol{m}$ 
	and a system of weights $\boldsymbol{\alpha}$ at the points of $D$. We also assume that the system of weights and multiplicities carry 
	certain symmetry conditions (cf. Definition \ref{def:weighted-flag-of-symmetric-type}), and that $\boldsymbol{\alpha}$ does not contain 
	$0$. Let $\mc{M}^{\boldsymbol{m,\alpha}}_{L(D)}$ denote the moduli space of parabolic symplectic stable bundles of rank $r$ on $X$, with 
	the symplectic form taking values in the line bundle $L(D)$, where $L(D)$ has the special parabolic structure (see Section 
	\ref{subsection:parabolic-symplectic-bundles}). This moduli space is a normal quasi-projective variety. We compute the Brauer 
	group of the smooth locus of the aforementioned moduli, denoted by $(\mc{M}^{\boldsymbol{m,\alpha}}_{L(D)})^{sm}$. Our main result is the 
	following.
	
	\begin{theorem}[{Theorem \ref{thm:brauer-group-of-parabolic-symplectic-moduli-concentrated-weights}
			and Corollary \ref{cor:brauer-group-arbitrary-generic-weights}}]
		Fix $D=\{p_1,p_2,\ \cdots,p_n\}$ and $r$ as above. The following statements hold:
		\begin{enumerate}
			\item If $\deg(L)$ is even, $$\Br((\mc{M}^{\boldsymbol{m,\alpha}}_{L(D)})^{sm})\simeq\dfrac{\bb{Z}}{\gcd(2,m_{_{p_1,1}},m_{_{p_1,2}},\cdots,m_{_{p_1,\ell(p_1)}},\cdots,m_{_{p_n,1}},\cdots, m_{_{p_n,\ell(p_n)}})}$$
			
			\item if $\deg(L)$ is odd,
			\begin{equation*}
				\Br((\mc{M}^{\boldsymbol{m,\alpha}}_{L(D)})^{sm})\simeq
				\begin{cases}
					0 & \text{if}\ \frac{r}{2}\geq3 \ \text{is odd},\\
					\dfrac{\bb{Z}}{\gcd(2,m_{_{p_1,1}},m_{_{p_1,2}},\cdots,m_{_{p_1,\ell(p_1)}},\cdots,m_{_{p_n,1}},\cdots, m_{_{p_n,\ell(p_n)}})} & \text{if}\ \frac{r}{2}\geq 3\ \text{is even}.
				\end{cases}
			\end{equation*}
		\end{enumerate}
	\end{theorem}
	
	Here is a brief outline of the main ideas of the proof. The symmetry conditions on the system of weights and multiplicities allow 
	us to relate $\mc{M}^{\boldsymbol{m,\alpha}}_{L(D)}$ with the moduli space of usual semistable symplectic vector bundles of rank $r$, where 
	the symplectic form now takes values in $L$. Using this, and the results from \cite{BiHol10}, where the authors determine the Brauer 
	group of the regularly stable locus of the latter moduli, we first prove the result when the system of weights are concentrated. 
	Finally, using the results of Thaddeus on wall-crossing for variation of weights \cite{Th96}, we extend our result to arbitrary 
	generic weights.
	
	\section{preliminaries}
	
	\begin{definition}\label{def:parabolic-bundles}
		Let $X$ be a smooth connected complex projective curve of genus $g$, with $g\,\geq\, 3$. Fix a
		finite subset $D\, \subset\, X$ of $n$ distinct
		points; these are referred to as 'parabolic points'. A \textit{parabolic vector bundle} of rank $r$ on $X$ is a
		vector bundle $E$ of rank $r$ together with the data of a weighted flag on the fiber at each $p\,\in\, D$:
		\begin{align*}
			E_{_p} \,=\, E_{_{p,1}}&\,\supsetneq\, E_{_{p,2}}\,\supsetneq\, \cdots \,\supsetneq\, E_{_{p,\ell(p)}}
			\,\supsetneq\, E_{_{p,\ell(p)+1}} \,=\,0\\
			0\,\leq\,& \alpha_{_{p,1}}\,<\,\alpha_{_{p,2}}\,<\,\cdots\,<\,\alpha_{_{p,\ell(p)}}\,<\,1.
		\end{align*}
		\begin{itemize}
			\item Such a flag is said to be of length $\ell(p)$, and the numbers $m_{_{p,i}}\,:=\, \dim E_{_{p,i}} -
			\dim E_{_{p,i+1}}$ are called the \textit{multiplicities} of the flag at $p$.
			
			\item The flag at $p$ is said to be \textit{full} if $m_{_{p,i}} \,=\,1$ for every $i$, in which case clearly $\ell(p) \,=\, r$. 
			
			\item The collection of real numbers $\boldsymbol{\alpha} :=\{(\alpha_{_{p,1}}<\alpha_{_{p,2}}<\cdots<\alpha_{_{p,\ell(p)}})\}_{p\in D}$ is called a system of \textit{weights}. 
			
			\item A \textit{parabolic data} consists of a collection $\{(E_{_{p,\bullet}},\alpha_{_{p,\bullet}})\}_{p\in D}$ of weighted flags as above. 
			
			\item We shall sometimes denote a system of multiplicities (respectively, a system of weights) by the bold symbol \textbf{m} (respectively, $\boldsymbol{\alpha}$), when there is no scope of any confusion. Also, we shall often denote a parabolic vector bundle simply by $E_*$ and suppress the parabolic data.
		\end{itemize}
	\end{definition}
	
	\begin{remark}\label{rem:special-structure}
		Let $E_*$ be a parabolic vector bundle of rank $r$ having the trivial weighted flag at each $p\,\in\, D$, i.e.,
		$\ell(p) \,=\, 1$ (so that $E_{_{p,2}} \,=\,0$) and $\alpha_{_{p,1}} \,=\,0$ is the single weight at each $p\,\in\, D$. In such a case, we say that $E_*$ has the \textit{special} parabolic structure, and we shall not distinguish between a vector bundle $E$ and the parabolic bundle $E_*$ having a special structure. 
	\end{remark}
	
	\begin{definition}\label{def:parabolic-morphism}
		Let $E_*$ and $E'_*$ be two parabolic vector bundles over $X$ with parabolic divisor $D$. A \textit{parabolic morphism} $f_* \,:\, E_* \,\longrightarrow\, E'_*$ is an ${\mathcal O}_X$--linear
		homomorphism $f \,:\, E \,\longrightarrow\, E'$ of the underlying vector bundles satisfying the condition
		$f_p(E_{_{p,i}})\,\subset\, E'_{_{p,j+1}}$ for every $\alpha_{_{p,i}}\,>\, \alpha'_{_{p,j}}$ for each $p\,\in\, D$, where $f_p$ is the map, induced by $f$, of fibers over $p$.
	\end{definition}
	
	\subsection{Parabolic vector bundles as filtered sheaves}\label{subsection:filtered-sheaves}\hfill\\
	
	To define parabolic tensor product and parabolic dual for parabolic vector bundles, it is crucial to view them as filtered sheaves, as follows. Given a parabolic vector bundle $E_*$ on $X$, Maruyama and Yokogawa associate to it a filtration $\{E_t\}_{t\in \bb{R}}$ parametrized by $\bb{R}$ \cite{MaYo92}. The filtration encodes the entire parabolic data. We recall from \cite{MaYo92} some properties of this filtration:
	\begin{enumerate}[(1)]
		\item The filtration $\{E_t\}_{t\in \bb R}$ is decreasing as $t$
		increases, in other words, $E_{t+t'}\,\subset\, E_t$ for all
		$t'\,>\,0$ and $t$;
		\item it is left-continuous, meaning there exists $\epsilon_t\,>\, 0$ such that the above inclusion of $E_t$ into $E_{t-\epsilon_t}$ is an isomorphism for all $t\,\in\, \bb R$,
		
		\item $E_{t+1} \,=\, E_t\otimes \mc{O}_X(-D)$ for all $t$,
		\item $E_0$ coincides with the vector bundle $E$ of $E_*$ ,
		\item for a finite interval $[a,\, b]$, the set of 'jumps' given by 
		$\{t\,\in\,[a,\, b]\ \mid \ E_{t+\epsilon}\,\subsetneq\, E_t \ \forall \ \epsilon\,>\,0\}$ is finite, and 
		\item the filtration $\{E_t\}_{t\in\bb R}$ has a jump at $t$ if and only if the fractional part $t-[t]$ is a parabolic weight for $E_*$.
	\end{enumerate}
	
	Parabolic morphisms between two parabolic vector bundles correspond to filtration-preserving morphisms between the corresponding filtered sheaves. We shall sometimes use this viewpoint of treating a parabolic bundle as a filtered sheaf, without explicitly mentioning it. 
	
	\subsection{Some remarks on parabolic dual and parabolic tensor product}\label{subsection:filtration-on-dual}\hfill\\
	
	There is a well-defined notion of parabolic dual and parabolic tensor product of two parabolic vector bundles on $X$. We shall not
	describe the parabolic tensor product here, and refer to \cite{Yo95} for the details on their construction. A particular case of parabolic duals, which will be used here, is described below. 
	
	Let $E_*$ be a parabolic vector bundle on $X$, which may be thought of as a filtered sheaf as described in
	Section \ref{subsection:filtered-sheaves}. There, using $(5)$, it follows that there are only finitely many jumps in the interval $[-1,\,1]$.
	Define $E_{t+}$ to be $E_{t+\epsilon}$, where $\epsilon\,>\,0$ is sufficiently small so that the sheaf $E_{t+\epsilon}$ is independent
	of $\epsilon$ (such $\epsilon$ exists due to $(5)$). Fix an $\epsilon\,>\,0$ so that $E_{t+} \,=\, E_{t+\epsilon}$
	for all $t\,\in\,[-1,\,1]$. If $t\,\in\, [0,\,1)$ is not a parabolic weight, then $E_{t+}$ coincides with $E_t$ by $(6)$. It can be
	shown that the underlying bundle of the parabolic dual $E_*^{\vee}$ is given by $(E_{\epsilon-1})^{\vee}$ (see \cite[p. 9341]{BiPi20}).
	
	Let $\boldsymbol{\alpha}$ be a system of weights such that $0\,\notin\,\boldsymbol{\alpha}$.
	Suppose the underlying vector bundle of $E_*$ is $E$. For the parabolic dual $E^{\vee}_*$ the following statements hold:
	\begin{align}
		(E^{\vee}_*)_0 \,=\, (E_{\epsilon -1})^{\vee}\,
		\simeq\, (E_{\epsilon}\otimes \mc{O}(D))^{\vee} =\, (E_{0+}\otimes \mc{O}(D))^{\vee}\, =\, (E_0\otimes \mc{O}(D))^{\vee}
		\,=\, E^{\vee} \otimes \mc{O}(-D);\label{eqn:underlying-bundle-of-parabolic-dual}
	\end{align}
	see Section \ref{subsection:filtered-sheaves} for the first equality and note that the above
	equality $(E_{0+}\otimes \mc{O}(D))^{\vee}\, =\, (E_0\otimes \mc{O}(D))^{\vee}$ holds because $0\, \notin\, \boldsymbol{\alpha}$. 
	Thus the underlying vector bundle for $E^{\vee}_*$ coincides with $E^{\vee}(-D)$ provided $0\,\notin\, \boldsymbol{\alpha}$.
	
	It is briefly recalled from \cite[\S~2.1.2]{KySuZh21} how the parabolic structure on $E^{\vee}_*$ is obtained. Take any $p\,\in\, D$. If
	the filtration for $E_p$ is given by 
	\begin{align*}
		E_{_p} \,=\, E_{_{p,1}} \,\supsetneq\, E_{_{p,2}}\,\supsetneq\,\cdots\,\supsetneq\, E_{_{p,\ell(p)}}\,\supsetneq\, E_{_{p,\ell(p)+1}}\,=\,0,
	\end{align*}
	then the filtration of $(E^{\vee}_*)_0 \,=\, E^{\vee}(-D)$ (see the discussion above) at $p$ is obtained by considering the surjections
	\begin{align*}
		E_p^{\vee}(-D)_{_{p}} \,=\, E_{_{p,1}}^{\vee}\otimes\mc{O}(-D)_{_{p}} \,\twoheadrightarrow\, E_{_{p,2}}^{\vee}\otimes\mc{O}(-D)_{_{p}}
		\,\twoheadrightarrow \,\cdots \,\twoheadrightarrow\, E_{p,\ell(p)}^{\vee}\otimes\mc{O}(-D)_{_{p}}
	\end{align*}
	and then taking their kernels. The weighted flag for $E^{\vee}_*$ at $p$ is as follows:
	\begin{align*}
		E^{\vee}(-D)_{_{p}} \, = \, E'_{_{p,1}}\, &\,\supsetneq\, E'_{_{p,2}}\, \supsetneq\,\cdots \, \supsetneq\, E'_{_{p,\ell(p)}}\, \supsetneq\, E'_{_{p,\ell(p)+1}} \, = \, 0\\
		\alpha'_{_{p,1}}\,&\,<\, \alpha'_{_{p,2}}\,<\, \cdots\, <\, \alpha'_{_{p,\ell(p)}}\, <\, \alpha'_{_{p,\ell(p)+1}} \,:=\, 1,
	\end{align*}
	where $E'_{_{p,j}}\,:=\, \left(\dfrac{E_{_p}}{E_{_{p,\ell(p)+2-j}}}\right)^{\vee}\otimes \mc{O}_X(-D)_{_{p}}\, = \,
	\Hom\left(\dfrac{E_{_p}}{E_{_{p,\ell(p)+2-j}}},\ \mc{O}_X(-D)_p\right)$,\, and\, $\alpha'_{_{p,j}}\,:=\, 1-\alpha_{_{p,\ell(p)+1-j}}$ for all
	$1\, \leq\, j \, \leq \, \ell(p)+1$.
	
	\subsection{Parabolic symplectic vector bundles}\label{subsection:parabolic-symplectic-bundles}\hfill\\
	
	Parabolic symplectic bundles over a curve were defined in \cite{BiMaWo11}, which will be briefly recalled. Take a parabolic line bundle $L_*$ on $X$, i.e., a parabolic vector bundle of rank $1$ in the sense of Definition \ref{def:parabolic-bundles}. Let $E_*$ be a parabolic vector bundle together with a parabolic morphism 
	\begin{align*}
		\varphi_* \,\,:\,\, E_*\otimes E_*\,\,\longrightarrow\,\, L_*.
	\end{align*}
	Tensoring both sides by the parabolic dual $E_{*}^{\vee}$ we get a parabolic morphism 
	\begin{align*}
		\varphi_*\otimes \text{Id}\ :\ E_*\otimes E_*\otimes E_*^{\vee}\ \longrightarrow\ L_*\otimes E_*^{\vee}.
	\end{align*}
	The trivial bundle $\mc{O}_X$ with the special parabolic structure (see Remark \ref{rem:special-structure}) is a sub-bundle of $E_*\otimes E_*^{\vee}$. Let 
	$$\widetilde{\varphi}_* \,\,:\,\, E_*\,\,\longrightarrow\,\, E_*^{\vee}\otimes L_*$$
	be the parabolic morphism defined by the composition of maps
	\begin{align*}
		E_* \,\simeq\, E_*\otimes \mc{O}_X\,\hookrightarrow\, E_*\otimes (E_*\otimes E_*^{\vee}) \,=\,
		(E_*\otimes E_*)\otimes E_*^{\vee} \,\xrightarrow{\,\,\, \varphi_*\otimes \text{Id}\,\,}\, L_*\otimes E_*^{\vee}.
	\end{align*}
	
	\begin{definition}\label{def:parabolic-symplectic-bundle}
		A \textit{parabolic symplectic vector bundle} on $X$ taking values in $L_*$ is a triple $(E_*,\,\varphi_*,\,L_*)$ as above, such that $\varphi_*$ is anti-symmetric, and the above parabolic morphism $\widetilde{\varphi}_*$ is an isomorphism
		of parabolic bundles. 
	\end{definition}
	
	Let $E_*$ be a parabolic vector bundle of rank $r$ and degree $d$ on $X$. Define the \textit{parabolic slope} of $E_*$
	to be (see Definition \ref{def:parabolic-bundles})
	\begin{align}\label{eqn:parabolic-slope}
		\mu_{par}(E_*)\,\,:=\,\, \dfrac{d+\sum_{p\in D}\sum_{i=1}^{\ell(p)}m_{_{p,i}}\alpha_{_{p,i}}}{r}\, \in\, {\mathbb R}.
	\end{align}
	Consider a parabolic symplectic vector bundle $(E_*,\ \varphi_*,\ L_*)$. As $E\otimes E$ is a sub-sheaf of the vector bundle underlying $E_*\otimes E_*$, the parabolic morphism $\varphi_*$ gives rise to an $\mc{O}_X$-linear map $\varphi_0:E\otimes E\longrightarrow L$, where $L$ is the underlying line bundle of $L_*$.
	
	Any algebraic sub-bundle $F$ of the underlying vector bundle $E$ gets equipped with an induced parabolic structure by restricting the flags and weights of $E_*$ to $F$. Let $F_*$ denote the resulting parabolic bundle.
	\vskip 0.7in
	\begin{definition}[{\cite[Definition 2.1]{BiMaWo11}}]\label{def:parabolic-symplectic-semistable-bundle}\mbox{}
		\begin{enumerate}
			\item Let $(E_*,\,\varphi_*,\,L_*)$ be a parabolic symplectic vector bundle (see Definition \ref{def:parabolic-symplectic-bundle}). A
			holomorphic sub-bundle $F$ of the underlying bundle $E$ is said to be \textit{isotropic} if $\varphi_0(F\otimes F) \,=\, 0$, where $\varphi_0$ is described as above.
			
			\item A parabolic symplectic vector bundle $(E_*,\,\varphi_*,\,L_*)$ is said to be \textit{parabolic semistable}
			(respectively, parabolic stable) if for all nontrivial isotropic sub-bundles $F\,\subset\, E$ we have
			\begin{align*}
				\mu_{par}(F_*)\,<\, (\text{respectively,}\ \ \leq)\ \ \mu_{par}(E_*),
			\end{align*}
			where $F_*$ has the above mentioned induced parabolic structure.
		\end{enumerate}
	\end{definition}
	Here, we need to restrict ourselves to isotropic sub-bundles, as the maximal parabolic subgroups of the symplectic group are precisely
	those that preserve an isotropic subspace.
	
	\section{The Setup}\label{section:setup}
	
	Let $X$ be a smooth connected projective curve over $\bb{C}$ of genus $g$, with $g\,\geq\, 3$. Fix a line
	bundle $L$, and also fix a reduced effective divisor $D$ on $X$. Consider
	a parabolic symplectic bundle $(E_*,\,\varphi_*,\,L(D))$, i.e., 
	\begin{align*}
		\varphi_* \,:\, E_*\otimes E_*\,\longrightarrow\, L(D),
	\end{align*}
	where the line bundle $L(D)$ is given the special parabolic structure (see Remark \ref{rem:special-structure}). We also assume that the system of weights for the parabolic structure does not contain $0$ (cf. \eqref{eqn:underlying-bundle-of-parabolic-dual}).
	
	Since $E\otimes E$ is a subsheaf of the underlying vector bundle $(E_*\otimes E_*)_0$ for the
	parabolic vector bundle $E_*\otimes E_*$, we get a map 
	$$\varphi\,:\, E\otimes E\,\longrightarrow\, L(D)$$ induced by $(\varphi_*)_0$.
	Moreover, the parabolic isomorphism $E_*\,\simeq\, E_*^{\vee}\otimes L(D)$ induced from $\varphi_*$ gives rise to an isomorphism $E\,\,\simeq\,\, (E^{\vee}_*)_0\otimes L(D)$
	of the underlying vector bundles. This, together with \eqref{eqn:underlying-bundle-of-parabolic-dual}, gives
	the following: $$E\ \simeq \ E^{\vee}\otimes \mc{O}(-D)\otimes L(D)\ \simeq \ E^{\vee}\otimes L.$$
	Thus $\varphi_*$ induces a non-degenerate bilinear form $
	\varphi \,:\, E\otimes E\,\longrightarrow\, L$,
	which is the restriction of $(\varphi_*)_0$ to the subsheaf $E\otimes E\,\subset
	\, (E_*\otimes E_*)_0$ (cf. \cite[Example 3.2]{Yo95}). Clearly $\varphi$ is anti-symmetric.
	Thus we have proved the following:
	
	\begin{lemma}\label{lem:induced-symplectic-form-on-underlying-bundle}
		A parabolic symplectic form $\varphi_* \,:\, E_*\otimes E_*\,\longrightarrow\, L(D)$ induces a symplectic form
		$\varphi \,:\, E\otimes E\,\longrightarrow\, L$ on the underlying parabolic vector bundle $E$ of $E_*$.
	\end{lemma}
	
	\begin{remark}\label{rem:unquely-determined-bilinear-map}
		Observe that $\varphi_*$ is uniquely determined by $\varphi$ due to the following:
		\begin{align}
			\mc{PH}\it{om}\left(E_*,\,\mc{PH}\it{om}(E_*,\,L(D))_*\right)&\,\subset\,
			\mc{H}\it{om}(E,\, \mc{PH}\it{om}(E_*,\,L(D))_0)\quad \text{\cite[p. 1782]{Bot10}}\nonumber\\
			&=\, \mc{H}\it{om}\left(E,\,\mc{PH}\it{om}(E_*,\,L(D))\right)
			\quad\text{\cite[Definition 3.2]{Yo95}}\nonumber\\
			&\subset\, \mc{H}\it{om}\left(E,\,\mc{H}\it{om}(E,\,L(D))\right)
			\quad\text{\cite[pp. 1782]{Bot10}};\nonumber
		\end{align}
		the last inclusion map sends the parabolic map $\varphi_*$ (seen as a parabolic map $E_*\,\longrightarrow\,
		\mc{PH}\it{om}(E_*,\,L(D))_*$) to the map $\varphi \,:\, E\otimes E\,\longrightarrow\, L\,\subset\, L(D)$.
	\end{remark}
	
	Fix an even positive integer $r$. We shall assume that the partial flags at the parabolic points $p
	\,\in\, D$ are of the following type:
	\begin{align}\label{eqn:condition-on-multiplicities}
		E_p\ &\,=\,E_{_{p,1}}\,\supsetneq\, E_{_{p,2}}\,\supsetneq\, \cdots\,\supsetneq\, E_{_{p,\ell(p)}}	\,\supsetneq\, E_{_{p,\ell(p)+1}} \,=\, 0\nonumber\\
		&\text{such that}\ \ m_{_{p,j}} \,=\, m_{_{p,\ell(p)+1-j}}\ \ \forall\ \ 1\,\leq\, j\,\leq\, \ell(p). 
	\end{align}
	One particular example of such flags are, of course, the full flags.
	
	As a motivation for the type of partial flags that are being considered in this paper, take a symplectic vector space
	$V$ of dimension $2m$ together with a partial flag consisting of isotropic subspaces
	\begin{align*}
		V\,=\, V_1\,\supsetneq\, V_2\,\supsetneq\, V_{3}\,\supsetneq\, \cdots\,
		\supsetneq\, V_{\ell} \,\supsetneq\, V_{\ell+1} \,=\,0.
	\end{align*}
	So $\dim V_2\,\leq\, m$, and hence $\ell\,\leq\, m+1$.	We can always extend such a flag by considering their
	annihilating subspaces:
	\begin{align}\label{eqn:partial-flag}
		V\,=\,V_{\ell+1}^{\perp}\,\supsetneq\, V_{\ell}^{\perp}\,\supsetneq\, V_{\ell-1}^{\perp}\,\supsetneq\, \cdots
		\,\supsetneq\, V_{2}^{\perp}\,\supset\, V_2\,\supsetneq\,
		\cdots\,\supsetneq\, V_3\,\supsetneq\,\cdots\,\supsetneq\, V_{\ell}\,\supsetneq\, V_{\ell+1}\,=\, 0.
	\end{align}
	If $V_2$ is a Lagrangian subspace (i.e., $\dim V_2 \,=\, m$), then $V_2^{\perp}\,=\, V_2$, which forces $\ell$
	in \eqref{eqn:partial-flag} to be \textit{odd}. On the other hand, if
	$V_2$ is not a Lagrangian subspace (i.e., $\dim V_2 \,<\, m$), then $V_2^{\perp}\,\supsetneq\, V_2$, which forces $\ell$ to
	be \textit{even}. In either case, the dimensions of the successive quotients of the resulting flag
	in \eqref{eqn:partial-flag} evidently satisfy conditions similar to \eqref{eqn:condition-on-multiplicities}. 
	
	The next proposition shows that flags of type as in \eqref{eqn:condition-on-multiplicities} 
	induce a certain symmetry on the system of weights as well.
	
	\begin{proposition}\label{prop:isotropic-flag-type-1-and-2-proposition}
		Let $(E_*,\,\varphi_*,\, L(D))$ be a parabolic symplectic vector bundle of rank $r$ such that the
		flags at each parabolic point are of type as in \eqref{eqn:condition-on-multiplicities}. The following
		are satisfied at each $p\,\in\, D$:
		\begin{enumerate}[(i)]
			\item \ \ $\alpha_{_{p,i}} \,=\, 1-\alpha_{_{p,\ell(p)+1-i}}$ \ \ for all $1\,\leq\, i\,\leq\, \ell(p)$.
			
			\item The flag at $E_p$ is isotropic, meaning that $E_{_{p,i}} \,=\, E_{_{p,\ell(p)+2-i}}^{\perp}$ \ \ for
			all $1\,\leq\, i\,\leq\, \ell(p)+1$.
		\end{enumerate}
	\end{proposition}
	
	\begin{proof}
		First the parabolic structure on $E^{\vee}_*\otimes L(D)$ will be described. Recall that the underlying
		vector bundle for $E^{\vee}_*$ is given by $E^{\vee}\otimes \mc{O}(-D)$\ (cf.
		\eqref{eqn:underlying-bundle-of-parabolic-dual}), and thus the underlying vector bundle
		for $E^{\vee}_*\otimes L(D)$ is given by $E^{\vee}\otimes L$. From the discussion in
		Section \ref{subsection:filtration-on-dual} it follows that at each parabolic point $p\,\in \,D$, the weighted
		flag for $E_{_{p}}^{\vee}\otimes L_{_{p}}$ is as follows:
		\begin{align}\label{eqn:dual-filtration}
			E^{\vee}_{_{p}}\otimes L_{_{p}} \ = \ E'_{_{p,1}}\ &\supsetneq\ E'_{_{p,2}}\ \supsetneq\ \cdots \ \supsetneq\ E'_{_{p,\ell(p)}}\ \supsetneq\ E'_{_{p,\ell(p)+1}} \ = \ 0\\
			\alpha'_{_{p,1}}\ &<\ \alpha'_{_{p,2}}\ <\ \cdots\ <\ \alpha'_{_{p,\ell(p)}}\ <\ \alpha'_{_{p,\ell(p)+1}} \,:=\, 1,
		\end{align}
		where $E'_{_{p,j}} \,:=\, \Hom\left(\dfrac{E_{_p}}{E_{_{p,\ell(p)+2-j}}},\ L_{_{p}}\right)$,\ and \
		$\alpha'_{_{p,j}} \,:=\, 1-\alpha_{_{p,\ell(p)+1-j}} \ \ \forall\ 1\ \leq\ j \ \leq \ \ell(p)+1$. 
		
		Proof of (i):\, Take $p\,\in\, D$. From the description of the parabolic structure on $E^{\vee}_*\otimes L(D)$
		in \eqref{eqn:dual-filtration},it follows that as $\varphi$ is a parabolic morphism,
		$\widetilde{\varphi}_p \,:\, E_p\,\longrightarrow\, E_p^{\vee}\otimes L_p$
		satisfies the property
		\begin{align}
			\widetilde{\varphi}_p(E_{_{p,i}})\,\subset\, E'_{_{p,\ell(p)+2-j}}\ =\ \Hom\left(\dfrac{E_{_{p}}}{E_{_{p,j}}},\,
			L_{_{p}}\right)\label{ez}
		\end{align}
		whenever $\alpha_{_{p,i}} \,>\, \alpha'_{_{p,\ell(p)+1-j}}\ =\ 1-\alpha_{_{p,j}}$\ (see Definition \ref{def:parabolic-morphism}). It follows that $\dim E_{_{p,i}} \,\leq\, r-\dim E_{_{p,j}}$ whenever $\alpha_{_{p,i}} \,>\, 1-\alpha_{_{p,j}}$. 
		
		Computing the dimension of both sides of \eqref{ez},
		\begin{align*}
			\sum_{s = i}^{\ell(p)}m_{_{p,s}}\ \leq\ r- \sum_{t=j}^{\ell(p)}m_{_{p,t}}\ 
			=\ \sum_{t=1}^{j-1}m_{_{p,t}}
		\end{align*}
		because $\sum m_{_{p,i}} \,=\,r$. This implies that
		\begin{align*}
			\sum_{s =1}^{\ell(p)+1 -i}m_{_{p,s}}\ \leq\ \sum_{t=1}^{j-1}m_{_{p,t}}
		\end{align*}
		because $m_{_{p,s}} \,=\, m_{_{p,\ell(p)+1-s}}$ for all $s$. Thus $\ell(p)+1-i\,\leq\, j-1$,\ 
		so that\ $i\,\geq\, \ell(p)+2-j$.
		
		As the parabolic weights form an increasing sequence, this implies that
		\begin{equation}\label{y1}
			\alpha_{_{p,i}}\ \,\geq\ \, \alpha_{_{p,\ell(p)+2-j}}
		\end{equation}
		whenever $\alpha_{_{p,i}} \,>\, 1-\alpha_{_{p,j}}$. Hence, if $$\alpha_{_{p,i}}\,>\,
		1-\alpha_{_{p,\ell(p)+1-i}}$$ for some $i$, setting
		$j\,=\,\ell(p)+1-i$ in \eqref{y1} \ it is deduced that
		$$\alpha_{_{p,i}}\,\geq\, \alpha_{_{p,\ell(p)+2-(\ell(p)+1-i)}} \,=\, \alpha_{_{p,i+1}},$$
		which is a contradiction.
		
		Therefore, it is deduced that
		\begin{align}\label{eqn:weight-relation-2}
			\alpha_{_{p,i}}\ \leq\ 1-\alpha_{_{p,\ell(p)+1-i}}
		\end{align}
		for all $i$.
		
		On the other hand, since $(\varphi)^{-1}$ is also a parabolic morphism, again
		using \eqref{eqn:dual-filtration} it follows that
		\begin{align*}
			(\widetilde{\varphi}_p)^{-1}\left(E'_{_{p,\ell(p)+1-j}}\right)\ =\ (\widetilde{\varphi}_p)^{-1} \left(\Hom\left(\dfrac{E_{_p}}{E_{_{p,j+1}}},\,L_{_p}\right)\right)\ \,\subset\ \, E_{_{p,i+1}}
		\end{align*}
		whenever $1-\alpha_{_{p,j}}\ =\ \alpha'_{_{p,\ell(p)+1-j}}\,> \, \alpha_{_{p,i}}$.
		Once more, computing the dimension of both sides it follows that
		\begin{align*}
			r-\sum_{s=j+1}^{\ell(p)+1-1}m_{_{p,s}}\ \leq \ \sum_{t = i+1}^{\ell(p)}m_{_{p,t}}.
		\end{align*}
		This implies that $\sum_{s=1}^{j}m_{_{p,s}}\ \leq\ \sum_{t =1}^{\ell(p)-i}m_{_{p,t}}$,
		because $m_{_{p,t}} \,=\, m_{_{p,\ell(p)+1-t}}$ for all $t$,
		and hence it follows that $j\,\leq\, \ell(p)-i$.
		
		As the parabolic weights form an increasing sequence, this implies that
		\begin{equation}\label{y2}
			\alpha_{_{p,j}}\ \leq\ \alpha_{_{p,\ell(p)-i}}
		\end{equation}
		whenever $1-\alpha_{_{p,j}}\,>\, \alpha_{_{p,i}}$.
		Hence, if $\alpha_{_{p,\ell(p)+1-i}}\, < \, 1-\alpha_{_{p,i}}$ for some $i$, then \eqref{y2} implies that 
		\begin{align*}
			\alpha_{_{p,i}}\ \leq\ \alpha_{_{p,\ell(p)-(\ell(p)+1-i)}} \ =\ \alpha_{_{p,i-1}},
		\end{align*}
		which is again a contradiction.
		This, combined with \eqref{eqn:weight-relation-2}, implies that $\alpha_{_{p,i}} = 1-\alpha_{_{p,\ell(p)+1-i}}$
		for all $1\,\leq\, i\,\leq\, \ell(p).$
		
		Proof of (ii):\, For each $i$, we have \ $\alpha_{_{p,i}}\ >\ \alpha_{_{p,i-1}} \ =\
		1-\alpha_{_{p,\ell(p)+2-i}}$\ \ by (i). Thus, the parabolic
		morphism $\widetilde{\varphi}$ satisfies the condition
		\begin{equation}\label{y3}
			\widetilde{\varphi}_{_{p}}(E_{_{p,i}})\ \ \subset\ \
			\Hom\left(\dfrac{E_{_{p}}}{E_{_{p,\ell(p)+2-i}}},\ \, L_{_{p}}\right)
		\end{equation}
		for all $i$ (see Section \ref{subsection:filtration-on-dual}).
		It is easy to check using \eqref{eqn:condition-on-multiplicities} that the two sides of \eqref{y3} have a common
		dimension, namely $\sum_{\ell=i}^{\ell(p)}m_{_{p,\ell}}$. Since $\widetilde{\varphi}_p$ is injective,
		it follows that
		\begin{align*}
			\widetilde{\varphi}_{_{p}}\ :\ E_{_{p,i}}\ \, \stackrel{\simeq}{\longrightarrow}\ \, \Hom\left(\dfrac{E_{_{p}}}{E_{_{p,\ell(p)+2-i}}},
			\,\,L_p\right),
		\end{align*}
		and thus $E_{_{p,i}}\, =\, E_{_{p,\ell(p)+2-i}}^{\perp}$. This completes the proof.
	\end{proof}
	
	Proposition \ref{prop:isotropic-flag-type-1-and-2-proposition} prompts the following definition.
	
	\begin{definition}\label{def:weighted-flag-of-symmetric-type}
		Let $r$ be a positive even integer. Fix a finite set of points $D$ on the curve $X$ and a subset
		of positive integers $\{\ell(p)\}_{p\in D}$ \ satisfying the condition $\ell(p)\ \leq\ r$ for all\ $p\,\in\, D$.
		Suppose that
		\begin{align*}
			\boldsymbol{m} \,\,=\,\, \left\{(m_{_{p,1}},\, m_{_{p,2}},\, \cdots,\, m_{_{p,\ell(p)}})_{p\in D}\right\},\ \
			\boldsymbol{\alpha}\,\,=\,\,\left\{\left(\alpha_{_{p,1}}\,<\,\alpha_{_{p,2}}\,<\,
			\cdots\,<\,\alpha_{_{p,\ell(p)}}\right)_{p\in D}\right\}
		\end{align*}
		are a system of multiplicities and weights on points of $D$ respectively (thus $\sum_{i=1}^{\ell(p)}m_{_{p,i}}
		\,=\, r$ for all $p\,\in\, D)$.
		\begin{itemize}
			\item We shall say that $\boldsymbol{m}$ is of \textit{symmetric type}, if $m_{_{p,j}} \,=\,
			m_{_{p,\ell(p)+1-j}}$ for all $p\,\in\, D$ and $1\,\leq\, j\,\leq\, \ell(p)$.
			
			\item We shall say that $\boldsymbol{\alpha}$ is of \textit{symmetric type}, if $\alpha_{_{p,j}}
			\,=\, 1-\alpha_{_{p,\ell(p)+1-j}}$ for all $p\,\in\, D$ and $1 \, \leq\, j \, \leq \, \ell(p)$.
		\end{itemize}
	\end{definition}
	
	\begin{proposition}\label{prop:symplectic-parabolic-implies-isotropic-flag-partial-flag-case}
		Let $(E,\,\varphi,\, L)$ be a symplectic vector bundle of rank $r$. Let $$\{E_{p,\bullet},\,\,
		\alpha_{p,\bullet}\,=\,\left(\alpha_{_{p,1}}\,<\,\alpha_{_{p,2}}\,<\,\cdots\,<\,\alpha_{_{p,\ell(p)}}\right)\}_{p\in D}$$
		be a system of weighted flags such that both the resulting system of multiplicities and weights are of symmetric type (see Definition \ref{def:weighted-flag-of-symmetric-type}).
		Consider the resulting parabolic bundle $E_*$. Then $\varphi$ produces a parabolic symplectic bundle $(E_*,\,\varphi_*,\, L(D))$
		if and only if the flag $\{E_{p,\bullet}\}_{p\in D}$ is isotropic with respect to $\varphi_p$ at each
		$p\,\in\, D$, meaning $E_{_{p,i}}\, =\, E_{_{p,\ell(p)+2-i}}^{\perp}$ for all $1\,\leq\, i\,\leq\, \ell(p)+1$.
	\end{proposition}
	
	\begin{proof}\mbox{}
		(1)\, $(\implies)$:\ This follows from Proposition \ref{prop:isotropic-flag-type-1-and-2-proposition}.
		
		(2)\, $(\impliedby)$:\ Using notation in \eqref{eqn:dual-filtration} and Definition \ref{def:parabolic-morphism},
		we need to check that the following implication holds:
		\begin{equation}\label{p2}
			\left(\alpha_{_{p,i}} \ >\ \alpha'_{_{p,j}}\ =\ 1-\alpha_{_{p,\ell(p)+1-j}}\right)\ \, \implies\ \,
			\left(\widetilde{\varphi}_{_{p}}(E_{_{p,i}})\ \subset\ E'_{_{p,j+1}}\ =\ 
			\Hom\left(\dfrac{E_{_{p}}}{E_{_{p,\ell(p)+1-j}}},\,L_{_p}\right)\right).
		\end{equation}
		Assume that $\alpha_{_{p,i}}\ > \ 1-\alpha_{_{p,\ell(p)+1-j}}$ for some indices $i,\,j$. From the assumption
		it follows that $1-\alpha_{_{p,\ell(p)+1-j}} \,=\, \alpha_{_{p,j}}$. This implies that $\alpha_{_{p,i}}\ > \ 
		\alpha_{_{p,j}}$, and
		thus $i\, >\, j$, as the weights form an increasing sequence. Hence $i\,\geq\, j+1$, and thus
		$\ell(p)+1-j\ \geq\ \ell(p)+2-i$, which in turn implies
		that $E_{_{p,\ell(p)+1-j}}\,\subseteq\, E_{_{p,\ell(p)+2-i}}$. Therefore,
		$$\Hom\left(\dfrac{E_{_{p}}}{E_{_{p,\ell(p)+2-i}}},\,\,L_{_{p}}\right)\ \, \subseteq\ \,
		\Hom\left(\dfrac{E_{_{p}}}{E_{_{p,\ell(p)+1-j}}},\,\,L_{_{p}}\right).$$
		Now, by assumption the flag at each $p\,\in\, D$ is isotropic, which implies that
		$E_{_{p,i}} \,=\, E_{_{p,\ell(p)+2-i}}^{\perp}$. Following the same notation as in \eqref{eqn:dual-filtration},
		this implies that
		\begin{align*}
			\widetilde{\varphi}_{_{p}}(E_{_{p,i}})=\Hom\left(\dfrac{E_{_{p}}}{E_{_{p,\ell(p)+2-i}}},\,L_{_{p}}\right)
			\ \subseteq\ \Hom\left(\dfrac{E_{_{p}}}{E_{_{p,\ell(p)+1-j}}},\,L_{_{p}}\right)
			\ =\ E'_{_{p,j+1}}
		\end{align*}
		whenever $\alpha_{_{p,i}} \ >\ \alpha'_{_{p,j}}$.
		Consequently, \eqref{p2} holds. Therefore, $\varphi$ produces a parabolic symplectic bundle
		$(E_*,\,\varphi_*,\,L(D))$ if the flag $\{E_{_{p,\bullet}}\}_{p\in D}$
		is isotropic with respect to $\varphi_p$ at each $p\,\in\, D$. This completes the proof of the proposition.
	\end{proof}
	
	\begin{lemma}\label{lem:symmetric-weight-estimates-partial-flag-case}
		Let $r$ be a positive even integer. Fix a finite subset of points $D$ of cardinality $n$ on the curve $X$. Let
		$\boldsymbol{m}$ be a system of multiplicities, and let $\boldsymbol{\alpha}$ be a system of weights on
		those points, so that both $\boldsymbol{m}$ and $\boldsymbol{\alpha}$ are of symmetric
		type (see Definition \ref{def:weighted-flag-of-symmetric-type}). Also, fix a positive integer $r'\, <\, r$
		and the following set of data for each $p\,\in\, D$: 
		\begin{itemize}
			\item a set of positive integers $\{\ell(p)\}_{p\in D}$ satisfying $\ell(p)\, \leq\, r$,
			\item a subset $I'(p)\,\subset\, \{1,\,2,\,\cdots,\,\ell(p)\}$, and
			\item a set of positive integers 
			$\{\ m'_{_{p,i}}\,\, \mid\, \ p\,\in\, D,\ i\,\in\, I'(p),\ m'_{_{p,i}}\,\leq\, m_{_{p,i}}\
			\forall\ i\,\in\, I'(p)\}$ satisfying $\sum_{j\in I'(p)}m'_{_{p,j}} \,=\, r'$.
		\end{itemize}
		Then the following equations hold:
		\begin{enumerate}[(i)]
			\item\hfill$\begin{aligned}[t]
				\sum_{i=1}^{\ell(p)}m_{_{p,i}}\alpha_{_{p,i}}\ =\ \frac{r}{2},
			\end{aligned}$\hfill\null\\
			\item\hfill$\begin{aligned}[t]
				\left|\frac{1}{r}\left(\sum_{p\in D}\sum_{i=1}^{\ell(p)}m_{_{p,i}}\alpha_{_{p,i}}\right)
				- \frac{1}{r'}\left(\sum_{p\in D}\sum_{j\in I'(p)}m'_{_{p,j}}\alpha_{_{p,j}}\right)\right|
				\ <\ \sum_{p\in D}\left(\frac{1}{2}-\alpha_{_{p,1}}\right).
			\end{aligned}$\hfill\null
		\end{enumerate}
		\textnormal{(Here, note that the condition} $\alpha_{_{p,1}}\,<\, \frac{1}{2}$ \ \textnormal{is ensured by
			the conditions}	\ 
		$\alpha_{_{p,\ell(p)}}\,=\,1-\alpha_{_{p,1}}$\ \ \textnormal{and}\ \
		$\alpha_{_{p,\ell(p)}}\, > \, \alpha_{_{p,1}}$\textnormal{.)}
	\end{lemma}
	
	\begin{proof}
		Proof of (i):\ Denote $\theta_p \,:=\, \sum_{i=1}^{\ell(p)} m_{_{p,i}}\alpha_{_{p,i}}$. We have
		$\alpha_{_{p,i}} \,=\, 1-\alpha_{_{p,\ell(p)+1-i}}$ and $m_{_{p,i}} \,=\, m_{_{p,\ell(p)+1-i}}$ by
		the assumption on the weights and multiplicities. Thus it follows that
		$$
		\theta_p \ = \ \sum_{i=1}^{\ell(p)}m_{_{p,i}} \left(1-\alpha_{_{p,\ell(p)+1-i}}\right)
		$$
		$$
		= \, \sum_{i=1}^{\ell(p)}m_{_{p,i}}-\sum_{i=1}^{\ell(p)}m_{_{p,i}}\alpha_{_{p,\ell(p)+1-i}}\,
		= \, r\ -\ \sum_{i=1}^{\ell(p)} m_{_{p,\ell(p)+1-i}}\cdot\alpha_{_{p,\ell(p)+1-i}}\,
		=\, r \ - \ \sum_{j=1}^{\ell(p)} m_{_{p,j}}\alpha_{_{p,j}}= \ r - \theta_p,
		$$
		which implies that $\theta_p \,=\, \dfrac{r}{2}$.
		
		Proof of (ii):\ We have
		$$
		\frac{1}{r}\left(\sum_{p\in D}\sum_{i=1}^{\ell(p)}m_{_{p,i}}\alpha_{_{p,i}}\right) - \frac{1}{r'}
		\left(\sum_{p\in D}\sum_{j\in I'(p)}m'_{_{p,j}}\alpha_{_{p,j}}\right)
		\ \underset{\text{by} \ (i)}{=\joinrel=\joinrel=}\ \ \
		\frac{1}{r}\left(\dfrac{nr}{2}\right) - \dfrac{1}{r'}\left(\sum_{p\in D}\sum_{j\in I'(p)}
		m'_{_{p,j}}\alpha_{_{p,j}}\right)
		$$
		$$
		<\ \frac{n}{2} - \dfrac{1}{r'}\left(\sum_{p\in D}\sum_{j\in I'(p)} m'_{_{p,j}}\alpha_{_{p,1}}\right)\ =\ \dfrac{n}{2} - \dfrac{1}{r'}\left(\sum_{p\in D}r' \alpha_{_{p,1}}\right)\quad\left[\text{since}\ \sum_{j\in I'(p)}m'_{_{p,j}}=r'\right]
		$$
		\begin{equation}\label{eqn:less-than-part-partial-flag-case}
			=\ \frac{n}{2}-\sum_{p\in D}\alpha_{_{p,1}}
			\ =\ \sum_{p\in D} \left(\frac{1}{2}- \alpha_{_{p,1}}\right).
		\end{equation}
		
		On the other hand,
		$$
		\frac{1}{r}\left(\sum_{p\in D}\sum_{i=1}^{\ell(p)}m_{_{p,i}}\alpha_{_{p,i}}\right) -
		\frac{1}{r'}\left(\sum_{p\in D}\sum_{j\in I'(p)}m'_{_{p,j}}\alpha_{_{p,j}}\right)\ \ \underset{\text{by} \
			(i)}{=\joinrel=\joinrel=}\ \ \ \frac{1}{r}\left(\dfrac{nr}{2}\right) - \dfrac{1}{r'}\left(\sum_{p\in D}
		\sum_{j\in I'(p)} m'_{_{p,j}}\alpha_{_{p,j}}\right)
		$$
		$$
		>\ \dfrac{n}{2} - \dfrac{1}{r'}\left(\sum_{p\in D}\sum_{j\in I'(p)} m'_{_{p,j}}\ \alpha_{_{p,\ell(p)}}\right)\ 
		=\ \dfrac{n}{2} - \dfrac{1}{r'}\left(\sum_{p\in D}r'\ \alpha_{_{p,\ell(p)}}\right)\ \
		\left[\text{since}\ \sum_{j\in I'(p)}m'_{_{p,j}}=r'\right]
		$$
		$$
		=\ \frac{n}{2}-\sum_{p\in D}\alpha_{_{p,\ell(p)}}\
		=\ \sum_{p\in D} \left(\frac{1}{2}- \alpha_{_{p,\ell(p)}}\right)\ 
		=\ \sum_{p\in D} \left(\frac{1}{2}-1+ \alpha_{_{p,1}}\right)
		\quad\left[\text{since}\ 1-\alpha_{_{p,1}} \,=\, \alpha_{_{p,\ell(p)}}\right]
		$$
		\begin{equation}\label{eqn:greater-than-part-partial-flag-case}
			=\ \sum_{p\in D}\left(\alpha_{_{p,1}}-\frac{1}{2}\right).
		\end{equation}
		Thus, from \eqref{eqn:less-than-part-partial-flag-case} and \eqref{eqn:greater-than-part-partial-flag-case} we conclude that 
		$$
		\left|\frac{1}{r}\left(\sum_{p\in D}\sum_{i=1}^{r}m_{_{p,i}}\alpha_{_{p,i}}\right) -
		\frac{1}{r'}\left(\sum_{p\in D}\sum_{j\in I'(p)}m'_{_{p,j}}\alpha_{_{p,j}}\right)\right|
		\ <\ \sum_{p\in D}\left(\frac{1}{2}-\alpha_{_{p,1}}\right).
		$$
		This completes the proof.
	\end{proof}
	
	\begin{definition}\label{def:concentrated-weight}
		Let $r$ be a positive even number. Fix a finite subset of points $D$ of $X$ and a set of positive integers
		$\{\ell(p)\}_{p\in D}$ satisfying the condition $\ell(p)\, \leq\, r$ for all $p\,\in\, D$.
		We shall say that a system of weights $$\boldsymbol{\alpha} \
		:=\ \left\{\left(\alpha_{_{p,1}}\,<\,\alpha_{_{p,2}}\,<\,\cdots\,<\,\alpha_{p,\ell(p)}\right)_{p\in D}\right\}$$
		is \textit{concentrated} if it of symmetric type (see Definition \ref{def:weighted-flag-of-symmetric-type}),
		and satisfies the inequality $\underset{p\in D}{\sum} \left(\frac{1}{2}-\alpha_{_{p,1}}\right) \, <\,
		\dfrac{1}{r^2}$. 
	\end{definition}
	
	\begin{lemma}\label{lem:concentrated-weight-parabolic-semistable-implies-semistable}
		Fix a positive even integer $r$, parabolic points $D$ on $X$, and a system of multiplicities $\boldsymbol{m}$
		of symmetric type (see Definition \ref{def:weighted-flag-of-symmetric-type}). Let
		$\boldsymbol{\alpha}$ be a concentrated system of weights (see Definition
		\ref{def:concentrated-weight}) compatible with $\boldsymbol{m}$ in the obvious sense. Then the following
		statements hold:
		\begin{enumerate}[(i)]
			\item If $(E_*,\,\varphi_*,\,L(D))$ is a parabolic symplectic semistable bundle of rank $r$ with system of
			multiplicities $\boldsymbol{m}$ and weights $\boldsymbol{\alpha}$, then the resulting symplectic bundle
			$(E,\varphi,L)$ is symplectic semistable (cf. Lemma \ref{lem:induced-symplectic-form-on-underlying-bundle}).
			
			\item If $(E,\,\varphi,\,L)$ is a symplectic stable bundle of rank $r$, and $\{E_{p,\bullet}\}_{p\in D}$ is a
			system of flags having multiplicities $\boldsymbol{m}$ such that $\{E_{{p,\bullet}}\}_{p\in D}$ is isotropic
			with respect to $\varphi_{_{p}}$ at each $p\,\in \,D$, meaning that $E_{_{p,i}}
			\,=\, E_{_{p,\ell(p)+2-i}}^{\perp}$ for all $1\,\leq\, i\,\leq\, \ell(p)+1$. Then the parabolic bundle
			$(E_*,\,\varphi_*,\,L(D))$ resulting from Proposition
			\ref{prop:symplectic-parabolic-implies-isotropic-flag-partial-flag-case} is parabolic symplectic stable.
		\end{enumerate}
	\end{lemma}
	
	\begin{proof}
		The idea of the proof has been inspired by \cite[Proposition 2.6]{AlGo18}.
		
		Proof of\, (i):\ Let $F$ be an isotropic sub-bundle of $E$ of rank $r_{_F}$ (see Definition
		\ref{def:parabolic-symplectic-semistable-bundle}). Consider the parabolic structure induced on $F$ by
		intersecting the flags for $E_p$ with $F_p$ for each $p\,\in\, D$. As a part of this data, at each $p\,\in\, D$
		we get a subset $I_F(p)\,\subset\, \{1,\,2,\,\cdots,\,\ell(p)\}$ consisting of those indices $j$ for which
		${\alpha}_{_{p,j}}$ is a parabolic weight for $F_{_p}$. Let $\boldsymbol{m}_F$ be the system of multiplicities
		induced by $\boldsymbol{m}$ on $F$.
		
		Now, as $(E_*,\,\varphi_*,\,L(D))$ is parabolic semistable (see Definition
		\ref{def:parabolic-symplectic-semistable-bundle}), for each nontrivial isotropic sub-bundle $F$ of $E$ as above,
		$$
		\dfrac{\deg(F)}{r_{_F}} \ - \ \dfrac{\deg(E)}{r} \ \leq\ \frac{1}{r}\left(\sum_{p\in D}\sum_{i=1}^{\ell(p)}
		m_{_{p,i}}\alpha_{_{p,i}}\right)\ -\ \frac{1}{r_{_F}}\left(\sum_{p\in D}\sum_{j\in 
			I_F(p)}({m_{_F}})_{_{p,j}}{\alpha}_{_{p,j}}\right)\quad[\text{cf.}\ \, \eqref{eqn:parabolic-slope}]
		$$
		$$
		<\ \ \sum_{p\in D}\left(\frac{1}{2}-\alpha_{_{p,1}}\right),\ \ \quad\text{by Lemma}\ \
		\ref{lem:concentrated-weight-parabolic-semistable-implies-semistable}
		$$
		$$
		< \ \frac{1}{r^2},\ \ \quad\text{as}\ \ \boldsymbol{\alpha}\ \ \text{is concentrated}.
		$$
		Thus, $$r\deg(F)\ -\ r_{_F}\deg(E)\ <\ \dfrac{rr_{_F}}{r^2}\ <\ 1. $$
		As the left-hand side is an integer, we conclude that 
		$$r\deg(F)\ -\ r_{_F}\deg(E)\ \leq\ 0$$ for every nontrivial isotropic sub-bundle
		$F\,\subset\, E$, so $(E,\varphi,L)$ is a symplectic semistable vector bundle.
		
		Proof of \, (ii):\ Continuing with the same notation as above, given an isotropic flag
		$\{{E_{_{p,\bullet}}}\}_{p\in D}$, the resulting parabolic symplectic vector bundle $(E_*,\,\varphi_*,
		\,L(D))$ is parabolic stable if and only if for every nontrivial isotropic sub-bundle $F\,\subset\, E$ the
		inequality
		\begin{align*}
			\dfrac{\deg(F)+\sum_{p\in D}\sum_{j\in I_F(p)}(m_{_F})_{_{p,j}}\alpha_{_{p,j}}}{r_{_F}}\ < \
			\dfrac{\deg(E)+\sum_{p\in D}\sum_{i=1}^{\ell(p)}m_{_{p,i}}\alpha_{_{p,i}}}{r}
		\end{align*}
		holds, or equivalently, if and only if
		\begin{align*}
			r\deg(F)\ -\ r_{_{F}}\deg(E) \ < \
			\sum_{p\in D}\left(r_{_{F}}\sum_{i=1}^{\ell(p)}m_{_{p,i}}\alpha_{_{p,i}}\ -\ 
			r\sum_{j\in I_F(p)}(m_{_F})_{_{p,j}}\alpha_{_{p,j}}\right).
		\end{align*}
		
		On the other hand, since $(E,\,\varphi,\,L)$ is semistable, every nontrivial isotropic sub-bundle $F
		\,\subset\, E$ yields $r\deg(F) - r_{_{F}}\deg(E) \, < \, 0$, and hence
		\begin{equation}\label{eq-1}
			r\deg(F)\ -\ r_{_{F}}\deg(E) \ \leq \ -1. 
		\end{equation}
		By Lemma \ref{lem:symmetric-weight-estimates-partial-flag-case} and the fact that
		$\boldsymbol{\alpha}$ is concentrated, we get that
		$$
		\left|\sum_{p\in D}\left(r_{_{F}}\sum_{i=1}^{\ell(p)}m_{_{p,i}}\alpha_{_{p,i}}\, -\,
		r\sum_{j\in I_F(p)}(m_{_F})_{_{p,j}}\alpha_{_{p,j}}\right)\right| \ < \
		rr_{_{F}}\sum_{p\in D}\left(\frac{1}{2}\, -\, \alpha_{_{p,1}}\right) \ < \ \dfrac{rr_{_F}}{r^2}\ <\ 1,
		$$
		and hence
		$$
		-1 \ < \ \sum_{p\in D}\left(r_{_{F}}\sum_{i=1}^{\ell(p)}m_{_{p,i}}\alpha_{_{p,i}}\ -\ 
		r\sum_{j\in I_F(p)}(m_{_F})_{_{p,j}}\alpha_{_{p,j}}\right) \ < \ 1.
		$$
		Thus, \eqref{eq-1} implies that 
		\begin{align*}
			r\deg(F)\, -\, r_{_{F}}\deg(E) \ \leq \ -1 \ < \ \sum_{p\in D}\left(r_{_{F}}
			\sum_{i=1}^{\ell(p)}m_{_{p,i}}\alpha_{_{p,i}}\, -\, r\sum_{j\in I_F(p)}(m_{_F})_{_{p,j}}\alpha_{_{p,j}}\right),
		\end{align*}
		and thus $(E_*,\,\varphi_*, \,L(D))$ is symplectic parabolic stable.
	\end{proof}
	
	\section{Brauer group of parabolic symplectic moduli}\label{section:brauer-group-of-parabolic-symplectic-moduli}
	\subsection{The case of concentrated weights}\hfill\\
	
	We finally come to our main goal of computing Brauer groups. Following the
	notation of Definition \ref{def:parabolic-bundles}, fix an even positive integer $r$, a subset $D \,=
	\,\{p_1,\,p_2,\,\cdots,\,p_n\}$ of $n$ points in $X$ and a line bundle $L$ on $X$. Fix a system of multiplicities
	$\boldsymbol{m} \,=\, \{(m_{_{p_i,1}},\,m_{_{p_i,2}},\,\cdots,\,m_{_{p_i,\ell(p_i)}})\, \mid \,\,
	p_i\,\in\, D\}$ of symmetric type (see Definition \ref{def:weighted-flag-of-symmetric-type}). We first consider the case of concentrated system of weights (Definition \ref{def:concentrated-weight}), and consider more general system of weights in the next subsection.
	
	Let $\boldsymbol{\alpha}$ be a concentrated system of weights compatible with $\boldsymbol{m}$ and not 
	containing $0$. Let $\mc{M}^{\boldsymbol{m,\alpha}}_{L(D)}$ denote the moduli space of parabolic symplectic 
	stable bundles $(E_*,\,\varphi_*,\, L(D))$ of rank $r$ on $X$, where $L(D)$ as before has the special
	structure (see Remark \ref{rem:special-structure}). Also, let $\overline{\mc{M}}_L$ denote the moduli space
	of semistable 
	symplectic bundles $(F,\,\psi,\,L)$ of rank $r$ on $X$. The condition that the symplectic form $\varphi_*$ takes 
	values in a fixed line bundle $L(D)$ actually fixes the determinant of $E$, and thus $\overline{\mc{M}}_L$ is 
	the moduli space of twisted semistable $\Sp(r,\bb{C})$-bundles. Let $\mc{M}^{rs}_L$ (respectively, 
	$\mc{M}^s_L$) be the open subset of $\overline{\mc{M}}_L$ consisting of regularly stable symplectic bundles 
	(respectively, stable symplectic bundles). Recall that a symplectic stable vector bundle $(E,\ \varphi,\ L)$ is said to be \textit{regularly stable} if, for any nonzero (meaning not identically zero) $\mc{O}_X$-linear morphism $g : E \ \longrightarrow\ E$ making the diagram 
	\begin{align*}
		\xymatrix{E\otimes E\ar[r]^{\varphi} \ar[d]_{g\otimes g} & L \\
			E\otimes E \ar[ru]_{\varphi}}
	\end{align*} 
	commute, must equal to multiplication by $\pm 1$. We have the chain of inclusions
	$$\mc{M}^{rs}_{L}\ \subset\ \mc{M}^{s}_L\ \subset\ \overline{\mc{M}}_L.$$
	As we have chosen a concentrated system of weights, by Lemma
	\ref{lem:concentrated-weight-parabolic-semistable-implies-semistable} there exists a morphism
	\begin{align}
		\pi_0\ :\ \mc{M}^{\boldsymbol{m,\alpha}}_{L(D)}\ \longrightarrow\ \overline{\mc{M}}_L.
	\end{align}
	Let $V \,:=\, \pi_0^{-1}(\mc{M}^s_L)$, and denote $\pi \,:=\, \pi_0|_V$. We thus have the following diagram: 
	\begin{align}\label{diagram-1}
		\xymatrix{
			V \ar@{^{(}->}[r] \ar[d]_{\pi} & \mc{M}^{\boldsymbol{m,\alpha}}_{L(D)} \ar[d]^{\pi_0}\\
			\mc{M}^s_L \ar@{^{(}->}[r] & \overline{\mc{M}}_L
		}
	\end{align}
	
	\begin{lemma}\label{lem:fibre-bundle}
		The map $\pi$ in \eqref{diagram-1} is a fiber bundle with fibers isomorphic to
		$\prod_{i=1}^{n}\Sp(r,\bb{C})/P_i$, where $\Sp(r,\bb{C})$ denotes the symplectic group, and $P_i$ is a
		parabolic subgroup consisting of block upper-triangular matrices whose blocks are of size
		$(m_{_{p_i,1}},m_{_{p_i,2}},\cdots,m_{_{p_i,\ell(p_i)}})$.
	\end{lemma}
	
	\begin{proof}
		This follows from Proposition \ref{prop:symplectic-parabolic-implies-isotropic-flag-partial-flag-case} and Lemma \ref{lem:concentrated-weight-parabolic-semistable-implies-semistable}. 
	\end{proof}
	
	Denote $U \,:=\, \pi^{-1}(\mc{M}^{rs}_L)$. Let $(\mc{M}^{\boldsymbol{m,\alpha}}_{L(D)})^{sm}
	\,\subset\, \mc{M}^{\boldsymbol{m,\alpha}}_{L(D)}$ be the smooth locus. As $\pi$ is a fibre bundle with smooth fibres and the base $\mc{M}^{rs}_L$ is a smooth open subset, it follows that $U\,\subset\, (\mc{M}^{\boldsymbol{m,\alpha}}_{L(D)})^{sm}$
	is a smooth open subset.
	
	\begin{lemma}\label{lem:codimension-estimate}
		The following bound on codimension holds:
		$$\codim_{_{\left(\mc{M}^{\boldsymbol{m,\alpha}}_{L(D)}\right)^{sm}}}
		\left((\mc{M}^{\boldsymbol{m,\alpha}}_{L(D)})^{sm}\setminus U\right)\ \geq\ 2.$$
	\end{lemma}
	
	\begin{proof}
		Denote $Y\,:=\, \mc{M}^{\boldsymbol{m,\alpha}}_{L(D)}$ and $Y^{sm}
		\,:=\, (\mc{M}^{\boldsymbol{m,\alpha}}_{L(D)})^{sm}$ for notational convenience. We have the diagram
		\begin{align}
			\xymatrix{U \ar@{^{(}->}[r] \ar[d]_{\pi|_{_U}} & V \ar[d]^{\pi}\\
				\mc{M}^{rs}_L \ar@{^{(}->}[r] & \mc{M}^s_L}
		\end{align}
		To prove the lemma, we need to consider two cases depending on whether $U \,=\, Y^{sm}\cap V$ or not.
		
		\textbf{Case I:}\ Assume that $U \,=\, Y^{sm}\cap V$. As $Y$ is normal, we have $\codim_{Y^{sm}}(Y\setminus Y^{sm})
		\,\geq\, 2$. The open subset $V\,\subset\, Y$ is also normal. Now, as $U \,=\,Y^{sm}\cap V$, it is
		the smooth locus of $V$, and thus
		\begin{align*}
			\codim_{V}(V\setminus U)\ \geq\ 2.
		\end{align*}
		This implies that
		\begin{align*}
			\codim_{(Y^{sm}\cup V)}((Y^{sm}\cup V)\setminus U)\ \geq\ 2,
		\end{align*}
		because the subset $V\,\subset\, (Y^{sm}\cup V)$ is open, and thus 
		\begin{align*}
			\codim_{Y^{sm}}(Y^{sm}\setminus U)\ \geq \ 2,
		\end{align*}
		because $(Y^{sm}\setminus U)\,\subset\, ((Y^{sm}\cup V)\setminus U)$ is open.
		
		\textbf{Case II:}\ Assume that $U\ \subsetneq\ (Y^{sm}\cap V)$.
		Consider the chain of open subsets
		$U\,\subsetneq\, Y^{sm}\cap V\,\subset\, V.$
		
		We will show that $\codim_V(V\setminus U) \,\geq\, 2$. For this, first note that
		since $\overline{\mc{M}}_L$ is a normal projective variety, and $\mc{M}^{rs}_L$ is precisely the smooth
		locus of $\overline{\mc{M}}_L$ \cite[Corollary 3.4]{BiHo12}, we have
		$\codim_{\overline{\mc{M}}_L}(\overline{\mc{M}}_L\setminus \mc{M}^{rs}_L)\geq 2$. This clearly implies 
		$\codim_{\mc{M}^s_L}(\mc{M}^{s}_L\setminus \mc{M}^{rs}_L)\,\geq\, 2$. As $\pi$ is a fibration,
		it now follows that $\codim_V(V\setminus U) \,\geq \,2$.
		
		Thus $\codim_{Y^{sm}\cap V}\left((Y^{sm}\cap V)\setminus U\right)\,\geq\, 2$ as well
		(here we are using that $U\,\subsetneq\, (Y^{sm}\cap V)$, so that $((Y^{sm}\cap V)\setminus U)$ is
		a nonempty open subset of $(V\setminus U)$). Now, $(Y^{sm}\cap V)\setminus U$ is a nonempty open subset
		of $Y^{sm}\setminus U$, and hence has the same dimension. Therefore, it follows that
		\begin{align*}
			\codim_{Y^{sm}}(Y^{sm}\setminus U) \,=\,\dim(Y^{sm}) - \dim (Y^{sm}\setminus U)
			\,=\, \dim(Y^{sm}\cap V) - \dim((Y^{sm}\cap V)\setminus U)
			\,\geq\, 2.
		\end{align*}
		This completes the proof.
	\end{proof}
	
	The Brauer group of $\mc{M}^{rs}_L$ has been computed in \cite{BiHol10}, which is briefly recalled.
	Let $\mathscr{M}^{rs}_{L}$ denote the moduli stack of regularly stable symplectic bundles
	on $X$ such that the symplectic form takes values in $L$. The map to the coarse moduli space
	$$h\,:\, \mathscr{M}^{rs}_{L}\,\longrightarrow \,\mc{M}^{rs}_{L}$$ is a $\mu_2$-gerbe. Let 
	\begin{align}\label{eqn:gerbe-class}
		\phi\,\ \in\,\ H^2_{\acute{e}t}(\mc{M}^{rs}_{L},\ \mu_2)
	\end{align}
	be the class of $h$. Consider the image $\iota_*(\phi)\,\in\, H^2_{\acute{e}t}(\mc{M}^{rs}_L,\,\bb{G}_m)$
	under the homomorphism defined using the inclusion map $\iota\,:\,\mu_2\subset\bb{G}_m$. The following
	statements hold (\cite[Corollary 6.5 and Proposition 8.1]{BiHol10}):
	\begin{enumerate}[(i)]\label{eqn:brauer-group-of-symplectic-moduli}
		\item If $\deg(L)$ is even, then\ $\Br(\mc{M}^{rs}_L) \,=\, \bb{Z}/2\bb{Z}$;
		
		\item if $\deg(L)$ is odd,
		\begin{equation*}
			\Br(\mc{M}^{rs}_L) \ =\
			\begin{cases}
				0 & \text{if}\ \frac{r}{2}\,\geq\,3 \ \text{is odd},\\
				\bb{Z}/2\bb{Z} & \text{if}\ \frac{r}{2}\,\geq\, 3\ \text{is even}.
			\end{cases}
		\end{equation*} 
	\end{enumerate}
	Furthermore, the generator for the above Brauer group is given by $\iota_*(\phi)$.
	
	On the other hand, there exists a projective Poincar\'e bundle $\bb{P}$ on $X\times \mc{M}^{rs}_L$ (see, e.g. 
	\cite{BiGo14}). Let $\bb{P}_x$ denote its restriction to $\{x\}\times \mc{M}_L^{rs}$ for a fixed point $x
	\,\in\, X$. 
	The class $\iota_*(\phi)$ as constructed above coincides with the class of the Brauer-Severi variety $\bb{P}_x$ 
	in $\Br(\mc{M}^{rs}_L)$, and thus the class of $\bb{P}_x$ generates $\Br(\mc{M}^{rs}_L)$.
	
	\begin{theorem}\label{thm:brauer-group-of-parabolic-symplectic-moduli-concentrated-weights}
		Fix a positive even integer $r$ and a finite subset of points $D=\{p_1,p_2,\cdots,p_n\}$ on $X$. Let $\boldsymbol{m}$ and $\boldsymbol{\alpha}$ be a system of multiplicities and weights of symmetric type at each point of $D$ (see Definition \ref{def:weighted-flag-of-symmetric-type}), such that $\boldsymbol{\alpha}$ is concentrated (see Definition \ref{def:concentrated-weight}) and does not contain $0$.
		Then the following statements hold:
		\begin{enumerate}[(i)]
			\item If $\deg(L)$ is even, $\Br\left((\mc{M}^{\boldsymbol{m,\alpha}}_{L(D)})^{sm}\right)
			\,\simeq\,\dfrac{\bb{Z}}{\gcd(2,m_{_{p_1,1}},m_{_{p_1,2}},\cdots,m_{_{p_1,\ell(p_1)}},\cdots,m_{_{p_n,1}},\cdots, m_{_{p_n,\ell(p_n)}})}$\\
			
			\item if $\deg(L)$ is odd,
			\begin{equation*}
				\Br\left((\mc{M}^{\boldsymbol{m,\alpha}}_{L(D)})^{sm}\right)\,\simeq\,
				\begin{cases}
					0 & \text{if}\ \frac{r}{2}\,\geq\,3 \ \text{is odd},\\
					\dfrac{\bb{Z}}{\gcd(2,m_{_{p_1,1}},m_{_{p_1,2}},\cdots,m_{_{p_1,\ell(p_1)}},\cdots,m_{_{p_n,1}},\cdots, m_{_{p_n,\ell(p_n)}})} & \text{if}\ \frac{r}{2}\geq 3\ \text{is even}.
				\end{cases}
			\end{equation*}
		\end{enumerate}
	\end{theorem}
	\begin{proof}
		Since $\Sp(r,\bb{C})$ is simply connected, by uniformization results it follows that $\mc{M}^{rs}_L$ is
		simply connected \cite[Corollary 3.10]{BiMuPa21}.
		As $\pi$ is a fiber bundle with fiber $\prod_{i=1}^{n}\Sp(r,\bb{C})/P_i$ for parabolic subgroups
		$P_i$ (see Lemma \ref{lem:fibre-bundle}), we have $\pi_*\bb{G}_m \,=\, \bb{G}_m$ while
		$R^1\pi_*\bb{G}_m$ is the constant sheaf with stalk $\Pic(\prod_{i=1}^{n}\Sp(r)/P_i)$.
		Moreover, $(R^2\pi_*\bb{G}_m)_{\textnormal{torsion}}\,=\,0$ (cf. \cite[Lemma 3.1]{BiDe11} for details). 
		
		Thus from the $5$-term exact sequence associated to the spectral sequence
		$$E_2^{p,q} \ =\ H^p(\mc{M}^{rs}_L,\, R^q\pi_*\bb{G}_m) \ \implies\ H^{p+q}(U,\, \bb{G}_m)$$
		we get the following exact sequence:
		\begin{align*}
			\cdots\, \longrightarrow\, \Pic\left(\prod_{i=1}^{n}\Sp(r,\bb{C})/P_i\right)
			\,\simeq \,\bigoplus_{i=1}^{n}\Pic\left(\Sp(r,\bb{C})/P_i\right) \, \overset{\theta}{\longrightarrow}
			\,\Br(\mc{M}^{rs}_L)\,\longrightarrow\, \Br(U)\,\longrightarrow\, 0.
		\end{align*}
		By Lemma \ref{lem:codimension-estimate} we have $\Br(U)\,\simeq\, \Br\left((\mc{M}^{\boldsymbol{m,\alpha}}_{L(D)})^{sm}\right)$ \cite{Ce19}, thus
		the above exact sequence becomes the following exact sequence: 
		\begin{align}\label{eqn:exact-sequence}
			\cdots\, \longrightarrow\, \Pic\left(\prod_{i=1}^{n}\Sp(r,\bb{C})/P_i\right) \,\simeq\,
			\bigoplus_{i=1}^{n}\Pic\left(\Sp(r,\bb{C})/P_i\right) \,\overset{\theta}{\longrightarrow}\,
			\Br(\mc{M}^{rs}_L)\,\longrightarrow\, \Br((\mc{M}^{\boldsymbol{m,\alpha}}_{L(D)})^{sm})\,\longrightarrow\, 0.
		\end{align}
		For each $1\,\leq\, i\,\leq\, n$, let $Q_i\,\in\, \mr{SL}(r,\bb{C})$ be a parabolic subgroup for which
		$P_i = \Sp(r,\bb{C})\cap Q_i$ (namely, $Q_i$ consists of block upper-triangular matrices in $SL(r)$
		of same block size as those of $P_i$). The inclusion maps $\rho_i\,:\, \mr{Sp}(r,\bb{C})/P_i
		\,\hookrightarrow\, \mr{SL}(r,\bb{C})/Q_i$ induce isomorphisms of Picard groups: 
		$$\rho_i^* \ :\ \Pic(\mr{SL}(r,\bb{C})/Q_i)\ \stackrel{\simeq}{\longrightarrow}\
		\Pic(\mr{Sp}(r,\bb{C})/P_i)
		$$
		for all $1\,\leq\, i\,\leq\, n$ (cf. \cite[\S~2]{PeTi22}).
		The generators of $\Pic(\mr{SL}(r,\bb{C})/Q_i)$ for each $i$ are known explicitly (cf. proof of \cite[Lemma 3.1]{BiDe11}). 
		We have
		$$
		\bigoplus_{i=1}^{n}\Pic(\Sp(r,\bb{C})/P_i)\ \simeq\ \bb{Z}^{\oplus N},
		$$
		where $N\,=\, \sum_{j=1}^{n}(\ell(p_j)-1)$. For each $2\,\leq\, j\,\leq\, \ell(p_i)$, define 
		\begin{align*}
			n_{_{i,j}} \ :=\ \sum_{k=j}^{\ell(p_i)} m_{_{p_{i},k}}.
		\end{align*}
		We have seen earlier that the class $[\bb{P}_x]$ generates $\Br(\mc{M}^{rs}_L)$. If $\zeta_{i,j}$ denote 
		the generators of \newline $\bigoplus_{i=1}^{n}\Pic(\Sp(r,\bb{C})/P_i)\,\simeq\, \bb{Z}^{\oplus N}$ as in 
		\cite[(5.9)]{BiDe11}, the map $\theta$ in \eqref{eqn:exact-sequence} sends $\zeta_{i,j}$ to $n_{_{i,j}}\cdot 
		[\bb{P}_x]$. This, together with the description of $\Br(\mc{M}^{rs}_L)$ just after \eqref{eqn:gerbe-class} 
		completes the proof.
	\end{proof}
	
	\subsection{The case of arbitrary generic weights}\hfill\\
	In order to address the situation where the system of weights $\boldsymbol{\alpha}$ is not concentrated, we make a few remarks regarding the construction of the moduli $\mc{M}^{\boldsymbol{m,\alpha}}_{L(D)}$. More generally, Let $G$ be a connected reductive algebraic group acting on a projective variety $Y$. In order to construct a GIT quotient of $Y$ under the action of $G$, one has to fix an ample $G$-linearization on $Y$. Various authors have studied how the GIT quotients vary as one varies the linearization, and the notions of chambers and walls can be made sense in the more general situation of the $G$--ample cone in the N\'eron-Severi group of $G$--linearized line bundles on $Y$ (\cite[Definition 0.2.1]{DolHu98}, \cite{Th96}). 
	
	Now, the moduli space $\mc{M}^{\boldsymbol{m,\alpha}}_{L(D)}$ has been constructed in \cite{WaWe24} under the exact same assumptions on the system of weights and multiplicities that we have considered here, namely that they are of symmetric type (cf. \cite[Definition 2.2]{WaWe24}). It is easy to see that although the authors in \cite{WaWe24} consider integer weights lying between $[0,K]$ for a fixed positive integer $K$, their notion matches exactly with ours upon division by the integer $K$.
	
	Fixing a system of rational weights amounts to fixing a polarization on a certain product of flag varieties for taking the GIT quotient by a suitable special linear group (cf. \cite[\S 3]{WaWe24}; see also \cite{BhRa89}). Thus, the set of all possible system of weights of symmetric type correspond to elements in the cone of ample linearized line bundles mentioned above (cf. \cite{DolHu98,Th96}). By the virtue of variation of GIT principles, this cone is separated by finitely many hyperplanes called \textit{walls}, and the connected components of these hyperplane complements are known as \textit{chambers}. The moduli space remains unchanged as long as the system of weights vary in inside a chamber. We shall call a system of weights as \textit{generic} if it is contained in a chamber. Now, since the collection of concentrated system of weights (see Definition \ref{def:concentrated-weight}) is clearly an open subset in this cone, and the intersections of walls are of codimension one, clearly there exists a concentrated system of weights inside the cone which is not contained in any wall, and thus there exists a \textit{generic} concentrated system of weights.
	
	Next, we show that the Brauer groups of the smooth locus of the parabolic symplectic moduli remain isomorphic when we cross a single wall in the ample cone. This will allow us to go from a generic and concentrated system of weights to arbitrary generic system of weights. A few auxiliary lemmas will be mentioned for this purpose.
	
	Let us denote $\malpha \,:=\, \mc{M}^{\boldsymbol{m,\alpha}}_{L(D)}$ and $\mbeta \,:=\,
	\mc{M}^{\boldsymbol{m,\beta}}_{L(D)}$ for notational convenience, and similarly denote by $\malphasm$ and
	$\mbetasm$ their respective smooth loci. Suppose $\boldsymbol{\alpha}$ and $\boldsymbol{\beta}$ be two generic systems of weights lying in two adjacent chambers separated by a single wall in the ample cone described above. 
	Using \cite[Theorem 3.5]{Th96}, there exist closed subschemes $Z_{\boldsymbol{\alpha}}\subset \malpha$ and $Z_{\boldsymbol{\beta}}\subset \mbeta$ along which the blow-ups are isomorphic, and moreover, the exceptional divisors are identified under the isomorphism. 
	Taking the complements of $Z_{\boldsymbol{\alpha}}$ and $Z_{\boldsymbol{\beta}}$ in their respective moduli, it immediately follows that there exist open subsets $\ualpha\ \subset\ \malpha$ and $\ubeta\ \subset\ \mbeta$, both having complements of codimension at least $2$, together with an isomorphism 
	\begin{align}\label{eqn:variation-of-weight-isomorphism}
		f \ : \ \ualpha\ \stackrel{\simeq}{\longrightarrow}\ \ubeta .
	\end{align}
	
	The next lemma is not strictly necessary for our purpose; we mention it for the sake of it being interesting in its own right.  
	\begin{lemma}\label{lem:picard}
		Let $\boldsymbol{\alpha}$ and $\boldsymbol{\beta}$ be two systems of generic weights. Then $\Pic(\malphasm)\simeq \Pic(\mbetasm)$.
	\end{lemma}
	\begin{proof}
		Consider the moduli stack of symplectic parabolic bundles of quasiparabolic type $\boldsymbol{m}$, which is a smooth algebraic stack by \cite[Lemma 3.2.2]{HS10}. The Picard group of this moduli stack has a uniform description for any system of weights \cite[Theorem 1.1]{LaSo97}. Restricting ourselves to the parabolic regularly stable locus (which has isomorphic Picard group by codimension reasoning; see \cite[Lemma C.1]{BiMuWe23} and \cite[Lemma 7.3]{BiHo12_2}) gives a $\mu_2$-gerbe from the moduli stack of parabolic regularly stable symplectic bundles to its coarse moduli space.
		Thus, the Picard group of the coarse moduli of parabolic regularly stable symplectic bundles is the kernel of the weight map given in 
		\cite[Lemma 4.4]{BiHol10}.
		Hence the Picard group of the parabolic regularly stable coarse moduli has a similar description irrespective of the weight. Since $\malphasm$ and $\mbetasm$ are precisely the parabolic regularly stable loci of $\malpha$ and $\mbeta$ respectively, this proves our claim.
	\end{proof}


\begin{theorem}\label{thm:adjecent-chamber-brauwer-group}
	Let $\boldsymbol{\alpha}$ and $\boldsymbol{\beta}$ be two systems of generic weights which lie in two adjacent chambers described above, which are separated by a single wall. Then $$\Br(\malphasm)\simeq \Br(\mbetasm).$$
\end{theorem}

\begin{proof} 
	By the remarks preceding Lemma \ref{lem:picard}, we can find open subsets
	$\ualpha\,\subset\, \malpha$ and $\ubeta\,\subset\, \mbeta$, both having complements of codimension at least $2$, together with an isomorphism $
	f :\ualpha\ \stackrel{\simeq}{\longrightarrow}\ \ubeta $ (see \eqref{eqn:variation-of-weight-isomorphism}).
	As $\malpha$ is irreducible, it follows that $(\malphasm\cap \ualpha)\,\neq\, \emptyset$. We shall consider two cases depending on whether $\malphasm$ is contained in $\ualpha$ or not.
	
	\textbf{Case I\ :}\ \ \ Assume that $\malphasm\,\subseteq \,\ualpha$. In this
	case $\malphasm$ is the smooth locus of $\ualpha$. As $f$ is an isomorphism, $f(\malphasm)$ is the smooth
	locus of $\ubeta$. Since the smooth locus of $\ubeta$ is $\mbetasm\cap\ubeta$, we get that 
	$f(\malphasm) \,=\, \mbetasm\cap\ubeta$.
	This implies that
	$$\mbetasm\setminus f(\malphasm) \ =\ \mbetasm\setminus\ubeta,$$
	and hence\ $$\codim_{\mbetasm}(\mbetasm\setminus f(\malphasm)) \,=\, \codim_{\mbetasm}(\mbetasm\setminus\ubeta)
	\,\geq\, \codim_{\mbeta}(\mbeta\setminus\ubeta)\,\geq\, 2 .$$
	Consequently,
	\begin{align*}
		\Br(\mbetasm)\ \simeq \ &\Br(f(\malphasm))\quad\text{\cite[Theorem 6.1]{Ce19}}\\
		\simeq\ &\Br(\malphasm).
	\end{align*}
	If $\mbetasm\,\subseteq \,\ubeta$, the same reasoning would again show that $\Br(\malphasm)
	\,\simeq\,\Br(\mbetasm).$
	
	Thus, we are left with the case where $\malphasm\,\not\subset\,\ualpha$ and $\mbetasm\,\not\subset\,\ubeta$.
	
	\textbf{Case II:} \ Assume that $\malphasm\,\not\subset\,\ualpha$ and $\mbetasm\,\not\subset\,\ubeta$.
	Again, we have $$\codim_{\malphasm}(\malphasm\setminus\ualpha) \ \geq\, \codim_{\malpha}(\malpha\setminus\ualpha)
	\ \geq\ 2,$$
	and thus $\Br(\malphasm)\,\simeq\, \Br(\malphasm\cap\ualpha)$.
	Of course, the same isomorphism holds if $\boldsymbol{\alpha}$ is replaced by $\boldsymbol{\beta}$.
	
	Now, the isomorphism $f$ takes the smooth locus of $\ualpha$ to the smooth locus of $\ubeta$, which are
	given by $(\malphasm\cap\ualpha)$ and $(\mbetasm\cap\ubeta)$ respectively. Thus,
	$$\Br(\malphasm)\simeq\Br(\malphasm\cap\ualpha)\ \simeq\ \Br(\mbetasm\cap\ubeta)\simeq\Br(\mbetasm).$$
	This proves the theorem.
\end{proof}

\begin{corollary}\label{cor:brauer-group-arbitrary-generic-weights}
	Theorem
	\ref{thm:brauer-group-of-parabolic-symplectic-moduli-concentrated-weights} remains valid for any arbitrary
	generic system of weights in the ample cone.
\end{corollary}

\begin{proof}
	Since there are only finitely many walls, we can arrange the collection of chambers in the ample cone 
	in a sequence, say $C_1,\, C_2,\, \cdots,\, C_N$, where $C_1$ contains a concentrated system of
	weights (see Definition \ref{def:concentrated-weight}), and for each $1\,\leq\, i\,<\,N$, the chambers
	$C_i$ and $C_{i+1}$ are separated by a single wall. Choose systems of generic weights
	$\boldsymbol{\alpha}_i$ from each $C_i$ such that $\boldsymbol{\alpha}_1$ is
	concentrated. Theorem \ref{thm:adjecent-chamber-brauwer-group} now completes the proof.
\end{proof}

\section*{Acknowledgements}

We are very grateful to the referee for helpful comments to improve the exposition. The first-named author is partially supported by a J. C. Bose Fellowship (JBR/2023/000003). The second-named author is supported by the DST-INSPIRE Faculty Fellowship (Research Grant No.: DST/INSPIRE/04/2024/001521), Ministry of Science and Technology, Government of India.

\subsection*{Data availability statement}
No data were used or generated.

\subsection*{Conflict of interest statement}
On behalf of all authors, the corresponding author states that there is no conflict of interest.

\end{document}